\documentclass[reqno, 12pt]{amsart}

\usepackage{amsmath,amssymb,amsthm,amsfonts}
\usepackage{fullpage}
\usepackage[colorlinks=true,linkcolor=blue,anchorcolor=blue,citecolor=red]{hyperref}
\allowdisplaybreaks 
\usepackage{cite}
\usepackage{colonequals}

\newtheorem{theorem}{Theorem}[section]
\newtheorem{lemma}[theorem]{Lemma}

\newtheorem{conjecture}[theorem]{Conjecture}
\theoremstyle{definition}

\theoremstyle{remark}

\newcommand{\Z}{\mathbb{Z}}

\newcommand{\R}{\mathbb{R}}

\renewcommand{\P}{\mathbb{P}}
\newcommand{\E}{\mathbb{E}}
\newcommand{\F}{\mathbb{F}}

\newcommand{\cA}{\mathcal{A}}
\newcommand{\cB}{\mathcal{B}}
\newcommand{\cC}{\mathcal{C}}

\newcommand{\on}{\operatorname}

\renewcommand{\mod}[1]{\on{mod}#1}

\newcommand{\set}[1]{\left\{#1\right\}}

\author{Huixi Li}
\address{School of Mathematical Sciences and LPMC, Nankai University, Tianjin 300071, China}
\email{lihuixi@nankai.edu.cn}

\author{Biao Wang}
\address{School of Mathematics and Statistics, Yunnan University, Kunming, Yunnan 650091, China}
\email{bwang@ynu.edu.cn}

\author{Chunlin Wang}
\address{School of Mathematical Sciences, Sichuan Normal University, Chengdu 610064, China}
\email{c-l.wang@outlook.com}

\author{Shaoyun Yi}
\address{School of Mathematical Sciences, Xiamen University, Xiamen, Fujian 361005, China}
\email{yishaoyun926@xmu.edu.cn}

\date{\today}

\makeatletter
\@namedef{subjclassname@2020}{\textup{}2020 Mathematics Subject Classification}
\makeatother

\title[]{On covering systems of polynomial rings over finite fields}

\begin{document}
	
\begin{abstract}
In 1950, Erd\H{o}s posed a question known as the minimum modulus problem on covering systems for $\mathbb{Z}$, which asked whether the minimum modulus of a covering system with distinct moduli is bounded. This long-standing problem was finally resolved by Hough in 2015, as he proved that the minimum modulus of any covering system with distinct moduli does not exceed $10^{16}$. Recently, Balister, Bollob\'as, Morris, Sahasrabudhe, and Tiba developed a versatile method called the distortion method and significantly reduced Hough's bound to $616,000$. In this paper, we apply this method to present a proof that the smallest degree of the moduli in any covering system for $\mathbb{F}_q[x]$ of multiplicity $s$ is bounded by a constant depending only on $s$ and $q$. Consequently, we successfully resolve the minimum modulus problem for $\mathbb{F}_q[x]$ and disprove a conjecture by Azlin. 
\end{abstract}

\maketitle

\section{Introduction and statement of results}

As mentioned in the preface of the renowned textbook ``Number Theory in Function Fields" \cite{Rosen2002}, elementary number theory focuses on the arithmetic properties of the ring of integers $\mathbb{Z}$ and its field of fractions, the rational numbers $\mathbb{Q}$. Notably, the ring of polynomials over a finite field $\mathbb{F}_q$, denoted as $\mathbb{F}_q[x]$, shares many properties with $\mathbb{Z}$, leading to the expectation that numerous theorems and conjectures for $\mathbb{Z}$ have analogous counterparts for $\mathbb{F}_q[x]$. Honorably, Sawin and Shusterman proved the twin prime conjecture \cite{SS2022_TW} and the quadratic Bateman-Horn conjecture \cite{SS2022_BH} over $\mathbb{F}_q[x]$ in quantitative form in 2022, Bender \cite{Bender2014} partially solved the Goldbach conjecture over $\mathbb{F}_q[x]$ in 2014, Weil \cite{Weil1948} proved the Riemann hypothesis for curves over finite fields in the 1940s, and Gauss \cite{Gauss1966} proved the prime number theorem for $\mathbb{F}_q[x]$ in the 1800s. 

The famous minimum modulus problem for $\mathbb{Z}$ was first proposed by Erd\H{o}s in 1950. After extensive research efforts by various mathematicians over the years, Hough achieved a significant breakthrough in 2015, providing a celebrated solution in his work \cite{Hough2015} by proving that the smallest modulus in any distinct covering system of $\mathbb{Z}$ is at most $10^{16}$. 
In this paper, we aim to completely solve the minimum modulus problem for $\mathbb{F}_q[x]$. In particular, we prove that the smallest degree of the moduli in any covering system of $\mathbb{F}_q[x]$ of multiplicity $s$ is bounded by a constant depending only on $s$ and $q$, as stated in Theorem~\ref{main_theorem} below. 

Let us begin by introducing the current research status of covering systems for $\mathbb{Z}$. A \textit{covering system} for $\mathbb{Z}$ is a finite collection of arithmetic progressions,  the union of which is the integers. Covering systems were first introduced by Erd\H{o}s \cite{Erdos1950}
in 1950 to answer an additive number theory question of Romanoff, in which Erd\H{o}s proved that there exists an arithmetic progression that contains no integers of the form $p + 2^k$ with $p$ a prime and $k$ a natural number. Subsequently, covering systems have not only become a significant research subject in their own right, but also found diverse applications in various fields. In his collections of open problems \cite{Erdos1957,  Erdos1963, Erdos1971, Erdos1973, EG1980}, Erd\H{o}s posed a number of problems concerning covering systems. One of them is the well-known minimum modulus problem: is there a uniform upper bound on the smallest modulus of all covering systems with distinct moduli? We say that a covering system with distinct moduli is a \textit{distinct covering system}. Building upon the work  \cite{FFSPY2007, Nielsen2009, Pari2013} of other mathematicians, Hough \cite{Hough2015} made a remarkable breakthrough in 2015 and showed that the minimum modulus of a distinct covering system can not be arbitrarily large. 
\begin{theorem}[{\cite[Theorem 1]{Hough2015}}]
The least modulus of a distinct covering system is at most $10^{16}$. 
\end{theorem}
Here are some more recent results on covering systems for $\mathbb{Z}$ following Hough's groundbreaking work. In 2022, Balister, Bollob\'{a}s, Morris, Sahasrabudhe, and Tiba \cite{BBMRS2022} managed to reduce the bound from $10^{16}$ to $616,000$ by introducing the \textit{distortion method}, a general technique used to bound the density of the uncovered set. For more details about this method, please refer to \cite{BBMST2020_2}. 
Soon after, Cummings, Filaseta, and Trifonov \cite{CFT2022} successfully reduced the bound to $118$ for distinct covering systems with squarefree moduli. Additionally, they proved in general the $j$-th smallest modulus in a minimal covering system with distinct moduli is bounded by some unspecified constant $C_j, j \geq 1$. This result was improved by Klein, Koukoulopoulos, and Lemieux \cite{KKL2022} later in the sense that they gave a specified bound for the $j$-th smallest modulus. Recently, Li, Wang, and Yi \cite{LiWangYi2023pre} develop a much more general distortion method by constructing a sequence of probability measures on a finite set with an inverse system and particularly apply it in the number field setting to obtain an analogue of \cite[Theorem~3]{KKL2022} by Klein et al.

Now, let us shift our focus to the discussion of $\mathbb{F}_q[x]$. The polynomial ring $\mathbb{F}_q[x]$ over the finite field $\mathbb{F}_q$ shares several properties with the ring of integers $\mathbb{Z}$. Specifically, both rings are principal ideal domains, possess finitely many units, infinitely many prime elements, and have a finite residue class ring for any non-zero ideal. Naturally, this leads us to inquire whether there exists a uniform upper bound on the smallest degree of congruence in a covering system of $\mathbb{F}_q[x]$. This problem is the analogue of Erd\H{o}s' minimum modulus problem for $\mathbb{F}_q[x]$ and intrigues us to study covering systems of $\mathbb{F}_q[x]$.

In 2011, Azlin \cite{Azlin2011thesis} conducted a specific study on covering systems of $\mathbb{F}_2[x]$. He provided numerous examples of covering systems for $\mathbb{F}_2[x]$, which can be found in \cite[(6.1)-(6.15)]{Azlin2011thesis}. Moreover, Azlin showed that there is no distinct covering system of $\mathbb{F}_2[x]$ that uses at most one polynomial of each degree $\geq 1$; see \cite[Theorem~6.2]{Azlin2011thesis}. In fact, it is not hard to see this kind of result also holds for the general cases $\mathbb{F}_q[x], q\geq 3$ by the density argument $\sum_{i = 1}^\infty \frac{1}{q^i}< 1$. Furthermore, Azlin formulated the following analogue of Erd\H{o}s' minimum modulus problem for $\mathbb{F}_2[x]$.\footnote{Interestingly, emulating Erd\H{o}s and Selfridge, Azlin offers \$3.00 for a disproof of Conjecture~\ref{Azlin conjecture}, Micah Milinovich offers \$5.00 for a proof, and Nathan Jones offers \$4.00 for either.}
 
\begin{conjecture}[{\cite[Conjecture 6.3]{Azlin2011thesis}}]\label{Azlin conjecture}
    In $\F_2[x]$, for any degree $D>0$, there exists a distinct covering system $\cC=\set{a_i(x)\mod{m_i(x)}}_{i=1}^k$
    for which $\deg m_i(x)\geq D$ for each $i=1,2,\dots,k$.
\end{conjecture} 

In general, the analogue of Erd\H{o}s' minimum modulus problem for $\mathbb{F}_q[x]$ is stated as follows.

\begin{conjecture}\label{main conjecture}
    In $\F_q[x]$, for any degree $D>0$, there exists a distinct covering system $\cC=\set{a_i(x)\mod{m_i(x)}}_{i=1}^k$
    for which $\deg m_i(x)\geq D$ for each $i=1,2,\dots,k$.
\end{conjecture}

In this work, we will disprove Conjecture~\ref{main conjecture}, showing that the following Theorem~\ref{main_theorem} holds. To state our main theorem we first introduce some notation. Let $\F_q$ be a finite field with $q$ elements. Let $\cA=\set{A_i\colon 1\leq i\leq n}$ be a finite collection of arithmetic progressions in $\F_q[x]$, where $A_i=a_i(x)+\langle d_i(x)\rangle, a_i(x)\in \F_q[x],d_i(x)\in \F_q[x]$, $d_i(x)$ is monic, and $\langle d_i(x)\rangle=d_i(x)\F_q[x]$ for $1\leq i\leq n$. For any $f(x)\in \F_q[x]$, we let $|f(x)|=q^{\deg f(x)}$ be the norm of $f(x)$ in $\F_q[x]$. Suppose $1<|d_1(x)|\leq |d_2(x)|\leq \cdots\leq|d_n(x)|$. Let 
\[
    m(\cA)=\max_{d(x)\in \F_q[x]}\#\{1\leq i\leq n\colon d_i(x)=d(x)\}
\]
be the multiplicity of $\cA$. We say $\cA$ is a \textit{covering system} of $\F_q[x]$ if $\F_q[x]=\bigcup_{i=1}^n(a_i(x)+\langle d_i(x)\rangle)$. Then the main result of our paper is stated in the following theorem.
\begin{theorem}\label{main_theorem}
Let $s$ be a positive integer. Let $\cA=\{a_1(x)+\langle d_1(x)\rangle, \ldots, a_n(x)+\langle d_n(x)\rangle\}$ be a covering system of $\mathbb{F}_q[x]$ of multiplicity $m(\cA) = s$. Then there exists an absolute constant $c>0$  such that
\[
\min_{1\leq i\leq n}\deg d_i(x)\leq 3(c+3\log_qs)\log(cs+3s\log_qs).
\]
\end{theorem}

Theorem~\ref{main_theorem} provides a solution to the analogue of Erd\H{o}s' minimum modulus problem for $\mathbb{F}_q[x]$ by taking $s=1$. It is comparable to Hough's result \cite{Hough2015} in the following sense.  An advantage of our main result is that it allows for flexibility in choosing the multiplicity $s$, which is not necessarily restricted to $1$ as in Hough's work \cite{Hough2015}, while a drawback is that the absolute constant $c$ has not been explicitly specified. For the proof of Theorem~\ref{main_theorem}, we follow the approaches of Theorem 3 in Klein et al.'s work \cite{KKL2022}. One could also try to find an explicit upper bound on the smallest modulus in a covering system with distinct moduli of $\F_q[x]$, following through computations similar as those done by Balister et al.\cite{BBMRS2022} but within this new setting.

As an application of Theorem~\ref{main_theorem},  we obtain the following theorem. 

\begin{theorem}
\label{thm_application}
	The degree of the modulus with the $n$-th smallest degree in a minimal
covering system of $\F_q[x]$ with distinct moduli is bounded.
\end{theorem}

Theorem~\ref{thm_application} is an analogue of Theorem 1.2 in the work \cite{CFT2022} of Cummings, Filaseta, and Trifonov. One may use an inductive argument similar to \cite[Theorem 1.2]{CFT2022} to get a proof of Theorem~\ref{thm_application}. Instead, we follow
the work \cite[Theorem 3]{KKL2022} of Klein, Koukoulopoulos, and Lemieux to give an alternative proof in Section~\ref{sec_application}. If one can  find a suitable generalisation of the following theorem of Crittenden and Vanden Eynden: any $n$ arithmetic progression that covers the first $2^n$ positive integers cover all the integers, then one may get a quantitative form of Theorem~\ref{thm_application} by a combinatorial argument similar to the proof in Section~\ref{sec_application}.

Our paper is organized as follows. In Section~\ref{distortion_method_section}, we introduce a variant of Balister et al.'s distortion method for polynomial rings over finite fields. In Section~\ref{some_tech_lemmas_section}, we apply some moment estimations and present some results on friable polynomials to prepare for the proof of our main theorem. In Section~\ref{proof_main_theorem_section}, we complete the proof of Theorem~\ref{main_theorem}. In Section~\ref{sec_application}, we prove Theorem~\ref{thm_application}.

\section{The distortion method for polynomial rings over finite fields}\label{distortion_method_section}

In this section we introduce the distortion method for polynomial rings over finite fields. We first introduce some notation. Let $\cA=\{a_1(x)+\langle d_1(x)\rangle, \ldots, a_n(x)+\langle d_n(x)\rangle\}$ be a finite collection of arithmetic progressions of multiplicity $s$ in $\F_q[x]$. And we let $Q(x)=\mathrm{lcm}[d_1(x),\dots, d_n(x)]$, and write $Q(x)=\prod_{i=1}^Jp_i^{\nu_i}(x)$ with $|p_1(x)|\leq |p_2(x)|\leq \cdots \leq |p_J(x)|$. Let $Q_j(x)=\prod_{i=1}^jp_i^{\nu_i}(x), 1\leq j \leq J$ and $Q_0(x)=1$. For $1 \leq j \leq J$ we let 
\[
\cB_j=\bigcup_{\substack{1\leq i \leq n\\ d_i(x)\mid Q_j(x), d_i(x)\nmid Q_{j-1}(x)}} \set{f(x)+\langle Q(x)\rangle\colon f(x)\in a_i(x) + \langle d_i(x)\rangle}.
\]
Let $\pi_j\colon\F_q[x]/\langle Q(x)\rangle\to \F_q[x]/\langle Q_j(x)\rangle$ be the natural projection and let $\pi_0$ be the trivial map, i.e., $\pi_0\colon\F_q[x]/\langle Q(x)\rangle\to \{0\}$. Let 
\[
F_j(f(x))=\set{\tilde{f}(x)\in \F_q[x]/\langle Q(x)\rangle\colon \pi_j(\tilde{f}(x))=\pi_j(f(x))}
\]
be the fibre at $f(x)\in \F_q[x]/\langle Q(x)\rangle$. Then we can observe that
\[
    \#F_j(f(x))=q^{\deg Q(x)- \deg Q_j(x)}.
\]
For any $f(x)\in \F_q[x]/\langle Q(x)\rangle$, define
\[
\alpha_j(f(x))=\frac{\# (F_{j-1}(f(x))\cap\cB_j)}{\#F_{j-1}(f(x))}.
\]

Let $\P_0$ be the uniform probability measure. Let $\delta_1,\dots,\delta_J\in[0,\frac12]$. Assuming we have defined $\P_{j-1}$, we define $\P_j$ as follows:
\[
\P_j(f(x))\colonequals \P_{j-1}(f(x)) \cdot \left\{
\begin{aligned}
	&\frac{1_{f(x)\notin \cB_j}}{1-\alpha_j(f(x))}, & \text{if }  \alpha_j(f(x))<\delta_j,\\
	&\frac{\alpha_j(f(x))-1_{f(x)\in \cB_j}\delta_j}{\alpha_j(f(x))(1-\delta_j)} , & \text{if }  \alpha_j(f(x))\geq \delta_j.
\end{aligned}
\right.
\]

For $k\in\Z_{\geq 1}$, we put
\[
M_j^{(k)}\colonequals\E_{j-1}[\alpha_j^k]=\sum_{f\in \F_q[x]/\langle Q(x)\rangle}\alpha_j^k(f(x))\P_{j-1}(f(x)).
\]

The following theorem, which is a variant of Balister et al.'s result \cite[Theorem 3.1]{BBMRS2022}, is key to the distortion method for polynomial rings over finite fields.
\begin{theorem}
	If
	\[
	\sum_{j=1}^J
	\min\left\{M_j^{(1)},\frac{M_j^{(2)}}{4\delta_j(1-\delta_j)}\right\}<1 ,
	\]
	then $\cA$ does not cover $\F_q[x]$. 
\end{theorem}
\begin{proof}
Interested readers can refer to the proof of \cite[Theorem 2.1]{LiWangYi2023pre}, which presents a proof in a more general setting. To make this paper self-contained, we provide a summary of some key steps here. 

By the construction of $\P_j$ and $\cB_j$, we have $\P_j(\cB_j) \leq \P_{j-1}(\cB_j)=\E_{j-1}[\alpha_j] = M_j^{(1)}$ for $1\leq j \leq J$. Moreover, by the inequality $\max\set{a-d,0}\leq a^2/4d$ for all $a,d>0$ we have 
\begin{align*}
    \P_j(\cB_j) & = \sum_{f(x) + \langle Q(x)\rangle \in \cB_j} \max \set{0,\frac{\alpha_j(f(x))-\delta_j}{\alpha_j(f(x))(1-\delta_j)}}\P_{j-1}(f(x)) \\
    & = \frac{1}{1-\delta_j}\sum_{f(x) + \langle Q(x)\rangle \in \cB_j} \max \set{0,\alpha_j(f(x))-\delta_j}\P_{j-1}(f(x)) \\
    & \leq \frac{1}{1-\delta_j}\sum_{f(x) + \langle Q(x)\rangle \in \cB_j} \frac{\alpha_j(f(x))^2}{4\delta_j}\P_{j-1}(f(x))\\
    & = \frac{\E_{j-1}[\alpha_j^2]}{4\delta_j(1-\delta_j)}\\
    & = \frac{M_j^{(2)}}{4\delta_j(1-\delta_j)}.
\end{align*}
Hence we have 
\begin{equation}\label{eqn_pjbj}
\P_j(\cB_j)\leq \min\set{M_j^{(1)},\frac{M_j^{(2)}}{4\delta_j(1-\delta_j)}}.
\end{equation}
Since it is easy to check $\P_j(\cB_j)=\P_{j+1}(\cB_j)=\cdots=\P_J(\cB_j)$ for $1 \leq j \leq J$, summing both sides of \eqref{eqn_pjbj} over $1\leq j \leq J$ and applying the assumption of the theorem yields 
\[
\sum_{j = 1}^J \P_J(\cB_j) < 1, 
\]
which completes the proof of the theorem. 
\end{proof}

\section{Some technical lemmas}\label{some_tech_lemmas_section}
In this section, we estimate the moments $M_j^{(1)}$ and $M_j^{(2)}$ and provide results on friable polynomials over finite fields. 

\subsection{Bounding the first and second moments}
First, we provide upper bounds for the $\alpha_j(f(x))$ term that shows up in the definition of $M_j^{(k)}$, where $f(x)\in \mathbb{F}_q[x]/\langle Q(x)\rangle$, $1\leq j \leq J$, and $1 \leq k \leq 2$.
\begin{lemma}\label{lem_alpha_estimate}
	For any $f(x)\in \F_q[x]/\langle Q(x)\rangle$ and $1\leq j \leq J$, we have
\[
\alpha_j(f(x))\leq \sum_{r=1}^{\nu_j} \sum_{g(x)\mid Q_{j-1}(x)} \sum_{\substack{1\leq i \leq n\\ d_i(x)=g(x)p_j^r(x)}} \frac{1_{f(x)\subseteq a_i(x)+\langle g(x)\rangle}}{|p_j(x)|^r}.
\]
\end{lemma}

\begin{proof}
Notice that $\# F_{j-1}(f(x))=q^{\deg Q(x)- \deg Q_{j-1}(x)}$. For $f(x) \in \F_q[x]/\langle Q(x)\rangle$, we write it as $f(x)=h(x)+\langle Q(x)\rangle$ for some $h(x)\in \F_q[x]$. For $a(x) \in \mathbb{F}_q[x]$, if $a(x)+\langle Q(x)\rangle \in F_{j-1}(f(x))\cap\cB_j$, then by the definition of $F_{j-1}(f(x))$ and $\cB_j$ the variable $a(x)$ will satisfy the following two conditions: (1) $a(x)\equiv h(x)\on{mod}{Q_{j-1}(x)}$ and (2) there exists some $1\le i \le n$ with $d_i(x)\mid Q_j(x), d_i(x)\nmid Q_{j-1}(x)$ such that $a(x) \equiv a_i(x)\on{mod}{d_i(x)}$. It follows that 
\begin{align}
    \alpha_j(f(x))&=\frac{\#(F_{j-1}(f(x))\cap\cB_j)}{\# F_{j-1}(f(x))} \nonumber\\
    &\leq\frac1{q^{\deg Q(x)-\deg Q_{j-1}(x)}}\sum_{\substack{1\leq i \leq n\\ d_i(x)\mid Q_j(x), d_i(x)\nmid Q_{j-1}(x)}}\sum_{\substack{a(x)+\langle Q(x)\rangle \in \mathbb{F}_q[x] / \langle Q(x)\rangle\\
a(x)\equiv h(x)\on{mod}{Q_{j-1}(x)}\\a(x) \equiv a_i(x)\on{mod}{d_i(x)}}}1.\label{lem_alpha_estimate_eqn1}
\end{align}

Since $d_i(x)\mid Q_j(x)$ and $d_i(x)\nmid Q_{j-1}(x)$, we may write $d_i(x)=g(x)p_j^r(x)$ for some $g(x)\mid Q_{j-1}(x)$ and $1\leq r\leq \nu_j$. Since $a(x)\equiv a_i(x)\mod{d_i(x)}$ and $a(x)\equiv h(x)\mod{Q_{j-1}(x)}$, we have $h(x)\equiv a_i(x)\mod{g(x)}$ and hence $f(x)\subseteq a_i(x)+\langle g(x)\rangle$. 

With $a(x)\equiv a_i(x)\mod{d_i(x)}$, we can deduce that $a(x)\equiv a_i(x)\mod{p_j^r(x)}$. Combining it together with  $a(x)\equiv h(x)\mod{Q_{j-1}(x)}$ and using the Chinese Remainder Theorem, we find that $a(x)$ lies in some congruence class mod $Q_{j-1}(x)p_j^r(x)$. In particular, there are $|Q(x)|/|Q_{j-1(x)}p_j^r(x)|$ choices for $a(x) \mod{Q(x)}$ lying in a congruence class mod $Q_{j-1}(x)p_j^r(x)$. Thus, we arrive at the following estimate on the double summation
\begin{align}
    &\sum_{\substack{1\leq i \leq n\\ d_i(x)\mid Q_j(x), d_i(x)\nmid Q_{j-1}(x)}}\sum_{\substack{a(x)+\langle Q(x)\rangle \in \mathbb{F}_q[x] / \langle Q(x)\rangle\\
a(x)\equiv h(x)\on{mod}{Q_{j-1}(x)}\\a(x) \equiv a_i(x)\on{mod}{d_i(x)}}}1 \nonumber\\
\leq & \sum_{r=1}^{\nu_j} \sum_{g(x)\mid Q_{j-1}(x)} \sum_{\substack{1\leq i \leq n\\ d_i(x)=g(x)p_j^r(x)}}  1_{f(x)\subseteq a_i(x)+\langle g(x)\rangle} \cdot \frac{|Q(x)|}{|Q_{j-1(x)}p_j^r(x)|}.\label{lem_alpha_estimate_eqn2}
\end{align}
Then the desired upper bound for $\alpha_j(f(x))$ follows immediately by plugging \eqref{lem_alpha_estimate_eqn2} into \eqref{lem_alpha_estimate_eqn1}. This completes the proof.
\end{proof}

To estimate the $\P_{j-1}(f(x))$ term in the definition of $M_j^{(k)}$, where $f(x)\in \mathbb{F}_q[x]/\langle Q(x)\rangle$, $1\leq j \leq J$, and $1 \leq k \leq 2$, we organize the representatives $f(x)$ based on their residues modulo a divisor $d(x)$ of $Q(x)$. Subsequently, we provide the following estimate. 
\begin{lemma}\label{lem_p_j_estimate}
For each $0\leq j\leq J, f(x)\in \F_q[x]$, and any monic polynomial $d(x)\in \F_q[x]$ such that $d(x)\mid Q(x)$, we have
\[
\P_j(f(x)+\langle d(x)\rangle)\leq q^{-\deg d(x)}\prod_{p_i(x)\mid d(x),\ i\leq j}\frac{1}{1-\delta_i}.
\]
\end{lemma}

\begin{proof} We follow the proof of \cite[Lemma 3.4]{BBMRS2022} by induction on $j$. Since $\P_0$ is the uniform probability measure, we have $\P_0(f(x)+\langle d(x)\rangle)=1/|d(x)|=q^{-\deg d(x)}$. So the estimate holds for the case $j=0$. Let $j\geq1$, and assume the estimate holds for $\P_{j-1}$. By the definition of $\P_j$, we notice that 
\[
\P_j(f(x)+\langle d(x)\rangle)\leq\frac{1}{1-\delta_j}\P_{j-1}(f(x)+\langle d(x)\rangle),\quad \forall \, f(x)\in \F_q[x].
\]

On one hand, suppose that $p_j(x)\mid d(x)$, then we have 
\begin{align*}
    \P_j(f(x)+\langle d(x)\rangle)&\leq \frac{1}{1-\delta_j}\P_{j-1}(f(x)+\langle d(x)\rangle)\\
    &\leq \frac{1}{1-\delta_j} q^{-\deg d(x)}\prod_{p_i(x)\mid d(x), i< j}\frac{1}{1-\delta_i}\\
    &=q^{-\deg d(x)}\prod_{p_i(x)\mid d(x), i\leq j}\frac{1}{1-\delta_i}.
\end{align*}

On the other hand, if $p_j(x) \nmid d(x)$, then $d(x)=m(x)\ell(x)$, where $m(x)=\gcd(d(x),Q_j(x))$ and $\gcd(\ell(x),Q_{j-1}(x))=1$. It follows that
\begin{align*}
    \P_j(f(x)+\langle d(x)\rangle)&=\frac{\P_j(f(x)+\langle m(x)\rangle)}{|\ell(x)|}=\frac{\P_{j-1}(f(x)+\langle m(x)\rangle)}{|\ell(x)|}\\
    &\leq \frac{1}{|\ell(x)||m(x)|} \prod_{p_i(x)\mid m(x), i< j}\frac{1}{1-\delta_i}=q^{-\deg d(x)}\prod_{p_i(x)\mid d(x), i\leq j}\frac{1}{1-\delta_i},
\end{align*}
as required.   
\end{proof}
By utilizing Lemma~\ref{lem_alpha_estimate} and Lemma~\ref{lem_p_j_estimate}, we can derive upper bounds for $M_j^{(1)}$ and $M_j^{(2)}$, where $1 \leq j \leq J$, provided that the parameters $\delta_i$, $1 \leq i \leq J$, are suitably chosen.
\begin{lemma}\label{lem_moments_estimate}
Let $s=m(\cA)$, let $\delta_1,\dots,\delta_J\in[0,\frac12]$, and let $j\in\{1, 2, \ldots, J\}$.
\begin{enumerate}
    \item If $\delta_i=0$ for $i\in\{1, \ldots, j-1\}$, then
    \[
        M_j^{(1)}\leq s\sum_{\substack{ \deg d(x)\geq \deg d_1(x)\\ \partial^+d(x)=\deg p_j(x)}}q^{-\deg d(x)}.
    \]
    Here, $\partial^+d(x)\colonequals\max\set{\deg p(x)\colon p(x)|d(x), p(x) \text{ is irreducible}}$ is the largest degree of the prime divisors of $d(x)$.
    \item We have
    \[
     M_j^{(2)}\ll \frac{s^2(\deg p_j(x))^{6}}{q^{2\deg p_j(x)}}.
    \]
\end{enumerate}    
\end{lemma}

\begin{proof} 
For any positive integer $k$, by Lemma~\ref{lem_alpha_estimate} we have
\begin{align}
    \E_{j-1}(\alpha_j^k)&=\sum_{f(x)\in \F_q[x]/\langle Q(x)\rangle}\alpha_j^k(f(x))\P_{j-1}(f(x))\nonumber\\
    &\leq \sum_{1\leq r_1,\dots, r_k\leq \nu_j} \sum_{g_1(x),\dots, g_k(x) \mid Q_{j-1}(x)} \sum_{\substack{1\leq i_1,\dots, i_k\leq n\\ d_{i_\ell}(x)=g_\ell(x) p_j^{r_\ell}(x), \forall \ell}} \frac{\P_{j-1}\left(\bigcap_{\ell=1}^k(a_{i_\ell}(x)+\langle g_\ell(x)\rangle)\right)}{|p_j(x)|^{r_1+\cdots+r_k}}.\label{kth moment}
\end{align}
We observe that $|g_\ell(x) p_j(x)^{r_\ell}|\geq |d_1(x)|$ for all $1\leq \ell\leq k$, since $|d_i(x)|\geq |d_1(x)|$ for all $1\leq i\leq n$. Moreover, given $r_1, \ldots, r_k$ and $g_1(x), \ldots, g_k(x)$, there are at most $s^k$ choices for $i_1, \ldots, i_k$ with $d_{i_\ell}(x)=g_\ell(x) p_j(x)^{r_\ell}$. By applying the Chinese Remainder Theorem, for each such choice of $i_1, \ldots, i_k$, the set $\bigcap_{\ell=1}^k(a_{i_\ell}(x)+\langle g_\ell(x)\rangle)$ is either empty or a congruence class modulo $\mathrm{lcm}[g_1(x), \ldots, g_k(x)]$. Using Lemma~\ref{lem_p_j_estimate} we can see that 
\begin{equation}\label{prob of intersection of APs}
    \P_{j-1}\left(\bigcap_{\ell=1}^k(a_{i_\ell}+\langle g_\ell(x)\rangle)\right) \leq \frac{1}{|\mathrm{lcm}[g_1(x), \ldots, g_k(x)]|}\prod_{\substack{i\leq j-1\\ p_i(x)\mid\mathrm{lcm}[g_1(x), \ldots, g_k(x)]}}(1-\delta_i)^{-1}
\end{equation}
for at most $s^k$ choices of $1\leq i_1,\dots, i_k\leq n$.
Then plugging \eqref{prob of intersection of APs} into \eqref{kth moment} we obtain
\begin{align*}
\E_{j-1}(\alpha_j^k)&\leq s^k \sum_{1\leq r_1,\dots, r_k\leq \nu_j} \sum_{\substack{g_1(x),\dots, g_k(x) \mid Q_{j-1}(x)\\ |g_\ell(x) p_j(x)^{r_\ell}|\geq |d_1(x)|, \forall \ell}} \frac{1}{|\mathrm{lcm}[g_1(x), \ldots, g_k(x)]|\cdot|p_j(x)|^{r_1+\cdots+r_k}}\\
&\qquad\qquad \cdot\prod_{\substack{i\leq j-1\\ p_i(x)\mid \mathrm{lcm}[g_1(x), \ldots, g_k(x)]}}(1-\delta_i)^{-1}.    
\end{align*}
In particular, when $k=1$ and $\delta_i=0$ for all $i\leq j-1$, it follows that
\begin{align*}
    M_j^{(1)} 
    & \leq s\sum_{1\leq r \leq \nu_j} \sum_{\substack{g(x)\mid Q_{j-1}(x)\\ |g(x)p_j(x)^r|\geq |d_1(x)|}} \frac{1}{|g(x)p_j(x)^r|}  \\
    & \leq s
    \sum_{\substack{\deg d(x)\geq \deg d_1(x)\\ \partial^+d(x)=\deg p_j(x)}}
    \frac{1}{|d(x)|}.
\end{align*}
This proves part (1) of the lemma. 

Next, let $k=2$ and $\delta_i\in[0,1/2]$ for all $1 \leq i \leq J$. It is clear that
\begin{equation*}
    \prod_{p_i(x)\mid \mathrm{lcm}[g_1(x), g_2(x)]}(1-\delta_i)^{-1}\leq 2^{\omega(\mathrm{lcm}[g_1(x), g_2(x)])}.
\end{equation*}
Here, $\omega(g(x))$ denotes the number of distinct monic irreducible polynomials factor of $g(x)$ for any $g(x)\in\F_q[x]$. Notice that $2^{\omega(g(x))}$ is a multiplicative function for $g(x)\in\F_q[x]$. Hence, we have
\begin{align*}
   M_j^{(2)} & \leq s^2\sum_{1\leq r_1, r_2\leq \nu_j}\sum_{g_1(x), g_2(x)\mid Q_{j-1}(x)}\frac{2^{\omega(\mathrm{lcm}[g_1(x), g_2(x)])}}{|\mathrm{lcm}[g_1(x), g_2(x)]| \cdot |p_j(x)|^{r_1+r_2}}\\
   &\leq \frac{s^2}{(|p_j(x)|-1)^2}\sum_{g_1(x), g_2(x)\mid Q_{j-1}(x)}\frac{2^{\omega(\mathrm{lcm}[g_1(x), g_2(x)])}}{|\mathrm{lcm}[g_1(x), g_2(x)]|}\\
   &\leq \frac{s^2}{(|p_j(x)|-1)^2}\prod_{i<j}\left(1+\sum_{\nu\geq 1}\frac{2}{|p_i(x)|^\nu} (2\nu+1)\right)\\
   &\ll \frac{s^2}{|p_j(x)|^2}\prod_{i<j}\left(1+\frac{6}{|p_i(x)|}+O\left(\frac{1}{|p_i(x)|^2}\right)\right)\\
   &\leq \frac{s^2}{|p_j(x)|^2}\,\mathrm{exp}\left\lbrace \sum_{i<j}\frac{6}{|p_i(x)|}+O\left(\frac{1}{|p_i(x)|^2}\right)\right\rbrace\\
   &\ll \frac{s^2(\log|p_j(x)|)^6}{|p_j(x)|^2}.
\end{align*}
The second last inequality holds due to  the fact that $1+t\leq e^t$ for all $t\in\R$. The last inequality in the above calculations is derived from Mertens’ estimate for function fields, i.e., $\sum_{|p(x)|\leq N}1/|p(x)|=\log\log N+O(1)$, which follows from the prime number theorem for function fields as shown in \cite[Theorem~(3.3.2)]{KnopfmacherZhang2001}. In conclusion, this completes the proof of part (2) of the lemma. 
\end{proof}

\subsection{A result on friable polynomials over finite fields}
First, we state the following counting result on friable polynomials over finite fields, which is essential for the proof of the main theorem. Let 
\begin{equation}
    \psi(n,m)\colonequals\#\set{f(x)\in \F_q[x] \text{ monic}\colon \deg f(x)=n, \partial^+ f(x)\leq m}.
\end{equation}
By \cite[Inequality (2)]{Warlimont1991} or \cite[Lemma 3.1]{Wang2022}, we have that $\psi(n,m)\ll q^ne^{-u/2}$ for $n\geq m\geq 3$ and $u=n/m$. Consequently, it follows that
\begin{lemma}\label{lem_keylem_smooth_poly}
For $n\geq m\geq3$ and $u=n/m$ we have
\[
\sum_{\substack{\deg d(x)\geq n\\  \partial^+d(x)\leq m}}q^{-\deg d(x)}\ll me^{-u/2}.
\]
\end{lemma}
\begin{proof} 
With $\psi(n,m)\ll q^ne^{-u/2}$, we have 
\[
\sum_{\substack{\deg d(x)\geq n\\ \partial^+d(x)\leq m}}q^{-\deg d(x)}=\sum_{\substack{v\geq n\\ d(x)\in \psi(v, m)}}q^{-v}
    \ll \sum_{v\geq n}q^v e^{-v/2m}q^{-v}
    =\sum_{v\geq n}e^{-v/2m}
    \ll me^{-u/2}.
\]
\end{proof}

\section{Proof of the Main Theorem}\label{proof_main_theorem_section}
To prove the main theorem, it suffices to show that there exists an absolute constant $c>0$ such that if
\begin{equation}\label{eqn_condition}
    \deg d_1(x)> 3(c+3\log_qs)\log(cs+3s\log_qs), 
\end{equation}
then $\cA$ does not cover $\F_q[x]$. In particular, it reduces to show that if \eqref{eqn_condition} holds for some sufficiently large $c$, then
\[
\eta\colonequals \sum_{j=1}^J
\min\left\{M_j^{(1)},\frac{M_j^{(2)}}{4\delta_j(1-\delta_j)}\right\}<1.
\]
Let $y=C+3\log_q s$, where $C$ is a large constant to be chosen later. Let $k=\max\{j\in [1, J]\cap \Z\colon \deg p_j(x)\leq y\}$ and let
\[
\delta_i=
\begin{cases}
0, &\text{if } i\leq k,\\
\frac{1}{2}, &\text{if } i> k.
\end{cases}
\]
We have 
\[
\eta\leq \sum_{1\leq j\leq k}M_j^{(1)}+\sum_{k<j\leq J}M_j^{(2)}=:\eta_1+\eta_2.
\]
Then, based on part (2) of Lemma~\ref{lem_moments_estimate}, it follows that
\begin{equation}\label{bound eta 2}
\begin{split}
    \eta_2&\ll \sum_{\deg p_j(x)>y}\frac{s^2(\deg p_j(x))^{6}}{q^{2\deg p_j(x)}}\ll \sum_{\deg p_j(x)> y}\frac{s^2}{q^{1.9\deg p_j(x)}}\\
    &\leq\sum_{n>y}\sum_{\substack{p(x) \text{ irreducible}\\\deg p(x)=n}}\frac{s^2}{q^{1.9n}}\ll \sum_{n>y}\frac{s^2}{q^{1.9n}}\frac{q^n}{n}\ll \sum_{n>y}\frac{s^2}{q^{0.9n}}\ll
    \frac{s^2}{q^{0.9y}}=\frac{1}{q^{0.9C}s^{0.7}}.
\end{split}
\end{equation}

Next, suppose $\deg d_1(x)>t$, where $t$ is to be chosen later. To bound $\eta_1$, we apply part (1) in Lemma~\ref{lem_moments_estimate} to obtain that
\begin{equation}\label{bound eta 1 first step}
    \eta_1\leq s\sum_{1\leq j\leq k}\sum_{\substack{ \deg d(x)\geq \deg d_1(x)\\ \partial^+d(x)=\deg p_j(x)}}q^{-\deg d(x)}\leq s\sum_{\substack{\deg d(x)\geq t\\ \partial^+d(x)\leq y}}q^{-\deg d(x)}.
\end{equation}
As regards the last sum above, we apply Lemma~\ref{lem_keylem_smooth_poly} and \eqref{bound eta 1 first step} to deduce that
\begin{equation}\label{bound eta 1 second step}
    \eta_1\ll sye^{-t/2y}.
\end{equation}

Now, by taking $t\colonequals 3y\log(sy)$, we can establish that $\eta_1\ll 1/\sqrt{sy}$. Combining \eqref{bound eta 2} and \eqref{bound eta 1 second step} above, we conclude that both $\eta_1$ and $\eta_2$ tend to $0$ as $C$ approaches $\infty$. Thus, we have $\eta<1$ for sufficiently large $C$. This completes the proof of Theorem~\ref{main_theorem}.\qed

\section{Proof of Theorem~\ref{thm_application}} 
\label{sec_application}
In this section, we prove Theorem~\ref{thm_application} by doing induction on $n$. 

By Theorem~\ref{main_theorem}, the case $n=1$ holds when we take $s=1$. Let $n>1$, and suppose that there are bounds $B_1,\dots,B_{n-1}$ such that the modulus with the $k$-th smallest degree in a minimal covering system of $\F_q[x]$ with distinct moduli is bounded by $B_k$  for each $1 \leq k < n$. 
	
Now, let $\cC=\set{a_i(x)+ \langle d_i(x)\rangle\colon 1 \leq i \leq \ell}$ be a minimal covering system of $\F_q[x]$  with distinct moduli and $\ell \geq n$ congruences, and $\deg d_1(x)\leq \cdots \leq \deg d_\ell(x)$. By the assumption, we have that $\deg d_k(x)\leq B_k$ for $1\leq k <n$. Let $L(x)=\on{lcm}[d_1(x), \dots, d_{n-1}(x)]$. Then $\deg L(x)\leq B\colonequals\sum_{k=1}^{n-1}B_k$. We write $\cC=\cC_1 \cup \cC_2$, where $\cC_1=\set{a_k(x)+ \langle d_k(x)\rangle\colon 1 \leq k <n}$ and $\cC_2=\set{a_i(x)+ \langle d_i(x)\rangle\colon n \leq i \leq \ell}$. Since $\cC$ is minimal, the progressions in $\cC_1$ do not form a covering system, hence the progressions in $\cC_2$
cover at least one congruence class modulo $L(x)$, say $g(x)+\langle L(x)\rangle$. Let
\[
\cC_3\colonequals\set{a_i(x)+r(x)+\langle d_i(x)\rangle\colon n \leq i \leq \ell, r(x)\in \F_q[x]/\langle L(x)\rangle},
\]
then $\cC_3$ is a covering system. Indeed, for any $a(x)\in \F_q[x]$
, there exists some $b(x)\in \F_q[x]/\langle L(x)\rangle$ such that $a(x)-b(x)\equiv g(x)\mod L(x)$. It follows that $a(x)-b(x)+\langle L(x)\rangle$ is covered by $\cC_2$. Thus, there exists some $n\leq i\leq \ell$ such that $a(x)-b(x)\equiv a_i(x)\mod d_i(x)$.  It follows that $a(x)\in a_i(x)+b(x)+\langle d_i(x)\rangle$. Hence $\cC_3$ is a covering system. Notice that the multiplicity of $\cC_3$ is at most $s=q^B$, and $d_n(x)$ is the polynomial with smallest degree in $\cC_3$. By Theorem~\ref{main_theorem}, we have $$\deg d_n(x) \leq 3(c+3\log_qs)\log(cs+3s\log_qs)=3(c+3B)\log(q^B(c+3B))$$ is bounded.  Thus, the desired result follows by induction. \qed

\section*{Acknowledgments}
The authors would like to thank the referee for very useful comments and insightful suggestions which led to improvements to this paper; in particular, for the suggestion of adding Theorem~\ref{thm_application} and its proof to the original version of the paper. Huixi Li's research is partially supported by the National Natural Science Foundation of China (Grant No. 12201313) and the Fundamental Research Funds for the Central Universities, Nankai University (Grant No. 63231145). Shaoyun Yi is supported by the National Natural Science Foundation of China (No. 12301016) and the Fundamental Research Funds for the Central Universities (No. 20720230025).


\end{document}